\documentclass[11pt]{article}
\usepackage{latexsym}
\usepackage{amssymb, amsmath,mathrsfs,amsthm}
\usepackage{geometry}
\usepackage{setspace}
\usepackage{mathtools}
\usepackage{mathtools}
\usepackage{graphics, epsfig}
\usepackage{lineno}
\usepackage{changes}

\usepackage{enumerate}
\usepackage{float}

\usepackage{color}

\definecolor{darkgreen}{rgb}{0,0.5,0}
\definecolor{darkred}{rgb}{0.8,0,0}
\definecolor{darkblue}{rgb}{0,0,0.7}

\newcommand{\lam}{z}
\newtheorem{theorem}{Theorem}[section]
\newtheorem*{theorem*}{Theorem}
\newtheorem{lemma}[theorem]{Lemma}

\newtheorem{corollary}[theorem]{Corollary}

\newtheorem{question}[theorem]{Question}

\theoremstyle{remark}
\newtheorem{remark}[theorem]{Remark}
\newtheorem{example}[theorem]{Example}

\usepackage{authblk}
\usepackage{varwidth}

\title{Crouzeix's Conjecture and Related Problems}

\author[1]{Kelly Bickel}
\author[2]{Pamela Gorkin
\thanks{Since August 2018, Pamela Gorkin has been serving as a Program Director in the Division of Mathematical Sciences
at the National Science Foundation (NSF), USA, and as a component of this position, she
received support from NSF for research, which included work on this paper. Any opinions,
findings, and conclusions or recommendations expressed in this material are those of the
authors and do not necessarily reflect the views of the National Science Foundation.}}
\author[3]{Anne Greenbaum}
\author[4]{Thomas Ransford
\thanks{This research is supported by grants from NSERC and the Canada Research Chairs program.}}
\author[5]{Felix Schwenninger}
\author[6]{Elias Wegert}
\affil[1]{Department of Mathematics, Bucknell University, Lewisburg, PA 17837 USA}
\affil[2]{Department of Mathematics, Bucknell University, Lewisburg, PA 17837 USA}
\affil[3]{Department of Applied Mathematics, University of Washington}
\affil[4]{D\'epartement de math\'ematiques et de statistique, Universit\'e Laval, Qu\'ebec (QC) G1V 0A6, Canada}
\affil[5]{Department of Mathematics, Universit\"at Hamburg, Germany and Department of Applied Mathematics, University of Twente, The Netherlands}
\affil[6]{Department of Mathematics and Computer Science,
Technische Universit\"at Bergakademie Freiberg}

\begin{document}

\maketitle

\centerline{In Memory of Stephan Ruscheweyh}
\begin{abstract}
In this paper, we establish several results related to Crouzeix’s conjecture.
We show that the conjecture holds for contractions with eigenvalues that are sufficiently well-separated. This separation is measured by the so-called separation constant, which is defined in terms of the pseudohyperbolic metric.
Moreover, we study general properties of related extremal functions and associated vectors.
Throughout, compressions of the shift serve as illustrating examples which also allow for refined results.
\end{abstract}

\bigskip
\section{Introduction}

\subsection{Motivation}
One of the most important results in operator theory is due to John von Neumann
and states that for a fixed contraction $T$ on a Hilbert space, the operator
norm satisfies
\[
\|p(T)\| \le \sup_{z \in \mathbb{D}} |p(z)| \quad \text{ for all } p \in
\mathbb{C}[z],
\]
where
$\mathbb{D}$ is the complex unit disk and
$\mathbb{C}[z]$ denotes the space of all one-variable polynomials with
complex coefficients.

Variations of von Neumann's inequality can be extremely useful and thus are
frequently the object of study. Matsaev's conjecture
(see \cite{N74}), for example, asserts that for every
contraction $T$ on $L^p(\Omega)$ (where $\Omega$ is a
measure space and $1 \le p \le \infty$) and $U: \ell^p \to \ell^p$ the
unilateral shift operator defined by $U(a_0, a_1, \ldots) = (0, a_0, a_1,
\ldots)$, the following inequality holds:
\[
\|p(T)\|_{L^p \to L^p} \le \|p(U)\|_{\ell^p \to \ell^p} \quad
\text{ for all } p \in \mathbb{C}[z].
\]
For $p = 1$ and $p = \infty$, it is not difficult to see that this is
true, and for $p = 2$ it is equivalent to
von Neumann's inequality. However, Drury \cite{D} showed that
Matsaev's conjecture fails
for $p=4$.

Von Neumann's inequality can be reformulated for a general bounded operator $T$ as
\[ \|p(T)\| \leq \sup_{|z|\leq \|T\|}|p(z)|\quad \text{ for all }p\in \mathbb{C}[z].\]
One may ask whether the supremum can instead be taken over subsets of $\{z\in\mathbb{C}: |z|\leq\|T\|\}$, such as the spectrum $\sigma(T)$, by possibly allowing for an absolute multiplicative constant $C$ in the inequality.  By Crouzeix's theorem \cite{Cr07}, we may choose the subset to be the numerical range $W(T)$ of $T$, but it is still an open question as to what the best multiplicative constant $C$ is. This problem is known as Crouzeix's conjecture. 
To state it precisely, let 
$A$ be an $n \times n$ matrix with complex entries. Let $W(A)$  denote its \emph{numerical range}
and $w(A)$ its numerical radius,
\[
W(A) = \left\{\langle Ax, x\rangle \in\mathbb{C}: \|x\| = 1 \right\},
\qquad
w(A) = \max \left\{|z|: z \in W(A) \right\},
\]
where $\langle\cdot,\cdot\rangle$ refers to the Euclidean inner product and $\|\cdot\|$ to its induced norm.
Then $W(A)$ always contains the \emph{spectrum} of $A$, denoted by $\sigma(A)$.
However, $W(A)$ often encodes significantly more information about $A$.
For example, $A$ is Hermitian if and only $W(A) \subseteq \mathbb{R}$. Further,
if $A$ is normal, then $W(A)$ is the convex hull of the eigenvalues of $A$.
Thus, it seems plausible that the value of $p$ on $W(A)$
could be used to control $\|p(A)\|$.\footnote{From now on $\| \cdot \|$
will denote the 2-norm for vectors and the corresponding operator norm for
matrices: $\| A \| = \sup_{\| x \| = 1} \| Ax \|$, which is also the largest
{\em singular value} of $A$.}
Indeed, in \cite{Cr04, Cr07}, Crouzeix showed that
\begin{equation}\label{eqn:cc}
\|p(A)\| \le C \sup_{z \in W(A)} |p(z)| \quad \text{ for all } p \in \mathbb{C}[z],
\end{equation}
with $C = 11.08$ and he \emph{conjectured} that the \emph{best constant} in
\eqref{eqn:cc} is $C=2$.
Recently, Crouzeix and Palencia~\cite{CP17} proved that
\eqref{eqn:cc} holds with $C=1+\sqrt{2}$, which is the best general constant
known so far.
Since every \emph{normal matrix} is unitarily equivalent to a diagonal matrix,
it is clear that
\eqref{eqn:cc} holds with $C=1$ in this case  
and the supremum can be taken over the eigenvalues only. 
It should be noted that if Crouzeix's conjecture holds for matrices,
then it automatically holds for bounded operators on every Hilbert space {\cite{Cr07}}.
Also, if it holds for all polynomials, then it holds for all
functions analytic in the interior of $W(A)$ and continuous on the boundary,
since such functions can be arbitrarily well approximated by polynomials
\cite{Mergelyan1,Mergelyan2}.

Crouzeix's conjecture has inspired a great deal of mathematics and there are now
several classes of matrices for which Crouzeix's conjecture has been proved
(see, for example, \cite{BCD,  CGL18, Ch13,  CG, Cr04,DD99, GC}). In
particular, the conjecture is true for $2 \times 2$ matrices as well as
matrices of the form $aI + DP$ or $a I + PD$ where $a$ is a complex number, $D$
is a diagonal matrix, and $P$ is a permutation matrix (see
\cite{Ch13}, \cite{Cr04}, and \cite{GKL}), $3 \times 3$
tridiagonal Toeplitz matrices and matrices in this class with some diagonal
entries taken equal to zero, \cite{GKL}.

It is easy to show that the conjecture holds for Jordan blocks with zeros along
the diagonal; it was later shown in \cite{CG} that the conjecture also holds
for perturbed Jordan blocks:
\[
J_\nu = \begin{pmatrix}
\lambda & 1& 0 & \cdots & 0\\
0 & \lambda & 1 & \cdots & 0\\
\vdots & \vdots & \ddots & \ddots  & \vdots\\
0 & 0 & \cdots & \lambda  & 1  \\
\nu & 0 & \cdots & 0 &\lambda\\
\end{pmatrix}.
\]
These two classes of matrices (Jordan blocks and perturbed Jordan blocks when
$\lambda = 0$) are special cases of operators known as \emph{compressions of
the shift operator}. To define a general compressed shift, let $H^2$ denote the
usual Hardy space
of holomorphic functions on
$\mathbb{D}$ and $S$ the shift operator defined
on $H^2$
by $(Sf)(z) = zf(z)$. Beurling showed that the (closed) nontrivial invariant
subspaces for $S$ are of the form $\Theta H^2$, where $\Theta$ is an inner
function. Therefore, the invariant subspaces of the adjoint, $S^{*}$, are of
the form $K_\Theta:=H^2 \ominus \Theta H^2$, where $\ominus$ denotes the
orthogonal complement. The associated compressed shift operator $S_\Theta:
K_\Theta \to K_\Theta$ is defined by $S_\Theta(f) = P_\Theta S|_{K_\Theta}.$ If
$\Theta= B$ is a finite Blaschke product with $\deg B =n$, that is, if
\[
B(z):= c \prod_{j = 1}^n \frac{z - a_j}{1 - \overline{a_j} z},
\quad a_1, \dots, a_n \in \mathbb{D},
\quad |c| = 1,
\]
then $K_{\Theta}$ is finite-dimensional and $S_{\Theta}$ can be represented as
an $n\times n$ matrix.  Because much is known about the numerical ranges of
these $S_{\Theta}$ (\cite{DGSV, GW, GW1, M}), it is natural to consider
Crouzeix's conjecture for them. Moreover, Sz.-Nagy and Foias \cite{SFBK10}
showed that every completely non-unitary contraction of class $C_{0}$ with defect $1$
is
unitarily equivalent to some $S_{\Theta}$. Thus, establishing Crouzeix's
conjecture for such $S_{\Theta}$ would imply it for a large collection of
matrices at once.

In their paper \cite{CP17},
Crouzeix and Palencia used clever complex analysis techniques to study
\eqref{eqn:cc}. Specifically, for $\Omega$ an open, convex
set with smooth boundary containing $W(A)$, they showed that for all $f$
holomorphic on  $\Omega$ and continuous up to the boundary,
\[
\| f(A) \| \le (1+\sqrt{2})\sup_{z \in \Omega} |f(z)|,
\]
which implies \eqref{eqn:cc} with $C = 1+\sqrt{2}$.
An $f$ maximizing $\|f(A)\|$ among all holomorphic functions
 $f$ on $\Omega$ with $\sup_{z\in\Omega}|f(z)|\le1$ is called
\textit{extremal} for the pair $(A,\Omega)$.
 By a normal-families argument, such an extremal $f$ always exists.
Recall that for a complex function $h$ defined and continuous on the boundary
of a set $\Omega$, the Cauchy transform of $h$ on $\Omega$ is defined by
\[
K(h)(z) = \frac{1}{2\pi i} \int_{\partial \Omega} \frac{h(\zeta)}{\zeta - z} \ d \zeta, \quad
z \in \Omega.
\]
A key part of the Crouzeix-Palencia proof showed that if $f \in \mathcal{A}(\Omega) := H^\infty(\Omega) \cap C(\overline{\Omega})$
and
$g= K(\overline{f})$, then
\[
\| f(A) + g(A)^{*} \| \le 2 \sup_{z\in \Omega} |f(z)|,
\]
where
the asterisk denotes the adjoint (or conjugate transpose)
of an operator.
It follows that if it were true that for extremal $f$, we have
\begin{equation}
\| f(A) \| \le \| f(A) + g(A)^{*} \|,
\label{strongerconj}
\end{equation}
then Crouzeix's conjecture would follow for $A$.
This motivates our study of such extremal $f$ below. See also \cite{RS18}
for an analysis of the Crouzeix-Palencia proof.

\subsection{Main Results}
In this paper, we provide a survey of our recent investigations related to
Crouzeix's conjecture; in particular, we derive specific bounds for $\|f(A)\|$,
where $f$ is chosen in an appropriate algebra, as well as properties of related
extremal functions and associated vectors. Throughout, we will use compressions
of the shift to both motivate and illustrate our results. While these
investigations have yielded a number of results, many questions remain open.
Below and throughout this paper, we will highlight these open questions and
invite any interested parties to take up their study.

Recall that if $A$ has distinct eigenvalues, then $A$ factors as $X \Lambda
X^{-1}$ for some diagonal $\Lambda$ and $f(A) = X f(\Lambda) X^{-1}$. In
Section \ref{sec:CC}, we use this formula paired with classical results about
function theory on $\mathbb{D}$ to study $\| f(A) \|$. First, in
Subsection~\ref{sec:cc1}, we let $A$ be a contraction with eigenvalues
$\lambda_1, \ldots, \lambda_n \in \mathbb{D}$ that are pseudohyperbolically
well separated (see \eqref{eqn:pseudo} below) and let $\delta$ denote a
constant depending on the separation of the eigenvalues given in
\eqref{eqn:delta1}. Then in
Theorem~\ref{crouzeix}, we combine results from interpolation theory with von
Neumann's inequality to deduce the existence of a constant $M(\delta)$ such
that
\begin{equation} \label{eqn:delta}
\|p(A)\| \le M(\delta) \max_{z \in \sigma(A)}|p(z)| \quad
\text{ for all } p \in \mathbb{C}[z],
\end{equation}
where $M(\delta) \to 1$ as $\delta \to 1$. For $\delta$ sufficiently close
to $1$, this implies that the matrix $A$ is \emph{near normal} in the sense
that it has a well-conditioned matrix of eigenvectors;
that is, 
\[\kappa (X) := \| X \| \cdot \| X^{-1} \|\]
is of moderate size.  We
thus have a criterion for near normality in terms of the eigenvalues and the
largest singular value (i.e., the operator norm) of the matrix.  Clearly, if
$\kappa (X) \leq 2$, then Crouzeix's conjecture holds for $A$ since
 \begin{equation}\label{eqn:kappa}
\| f(A) \| = \| X f( \Lambda ) X^{-1} \| \leq \kappa (X) \max_{z \in \sigma
(A)} | f(z) | \leq \kappa (X) \max_{z \in W(A)} | f(z) |.
\end{equation}
For a matrix $A$ with distinct eigenvalues, one can similarly define the minimal condition number of an eigenvector matrix of $A$ by
\begin{equation} \label{eqn:k}
\eta(A):=\inf\{\|X\| \|X^{-1}\|: A = X \Lambda X^{-1}\}.
\end{equation}
and as in \eqref{eqn:kappa}, if $\eta(A) \le 2$, then
Crouzeix's conjecture holds for $A$. In Subsection~\ref{sec:cc2}, 
we study this setup for
matrices $M_\Theta$ that are representations of the
compression of the shift $S_{\Theta}$, associated with a finite Blaschke product
$\Theta$ with distinct zeros in $\mathbb{D}$.
First, in Theorem~\ref{thm:eigenvectors}, we provide tractable formulas for specific
$X$ and $X^{-1}$. Then, in Theorem~\ref{thm:condition}, we use these formulas
to obtain a bound on $\eta (M_{\Theta})$ in terms of the separation of the
zeros of $\Theta$. Because compressions of the shift are quite important, we
pose the following question:

\begin{question} \label{q:condition}
What is the minimal condition number of an eigenvector matrix
$X$ of $M_{\Theta}$ for a general finite Blaschke
product $\Theta$ with distinct zeros?
\end{question}

In Section \ref{sec:extremal}, we turn to extremal functions and vectors. Let
$\Omega$ be a bounded simply connected domain with smooth boundary containing
the spectrum, $\sigma(A)$, of $A$. We are interested in studying
$\sup_f\|f(A)\|$, where the supremum is taken over all $f\in H_1^\infty(\Omega)$,
the closed
unit
ball of bounded holomorphic functions on $\Omega.$ As before, an $f$ for which
the supremum is attained is called \emph{extremal} for the pair $(A,\Omega)$
and any non-zero vector $x$ where $\| f(A) x \| = \| f(A)\| \| x\|$
(i.e., any right singular vector of $f(A)$ associated with the
largest singular value) is called an associated \emph{extremal vector}.
Such vectors are also called \emph{maximal}, see for example \cite{Sar67}.
Crouzeix \cite[Theorem~2.1]{Cr04} showed that an extremal
function
for $(A,\Omega)$ is necessarily of the form
$B_A\circ\phi$, where $\phi$ is a bijective conformal map of $\Omega$ onto the
open unit disk $\mathbb{D}$, and $B_A$ is a Blaschke product of degree at most
$n-1$.

In Section \ref{sec:comp}, we consider compressions of the shift $S_{\Theta}$
for $\Theta$ a finite Blaschke product and $\Omega= \mathbb{D}$. Theorem
\ref{thm:norm}, which is proved in \cite{GR09}, characterizes the extremal
functions for $(S_{\Theta}, \mathbb{D})$; here, we provide a simple proof in
the case where $\Theta$ has distinct zeros. Meanwhile, Theorem \ref{thm:vector}
characterizes the associated extremal vectors. In Section \ref{sec:degree}, we
consider a general $n\times n$ matrix $A$ and set $\Omega = W(A)^\circ$,
assuming that $\sigma(A) \subseteq W(A)^\circ$. Characterizing the extremal
functions in this situation is significantly more complicated. Instead of
tackling that problem in its entirety, we investigate the possible degrees of
an extremal Blaschke product $B_A$. In Section~\ref{sec:degree},
we give an example of a matrix (defined in \eqref{E:Crabb}) for which the only
extremal functions have
(maximal) degree $n-1$,
and in Theorem~\ref{T:openset}
we show
that there is an open set of $n\times n$ matrices for which the extremal
Blaschke products have maximal degree. The following question is still open:

\begin{question} \label{q:degree} Given an $n \times n$ matrix $A$ with
$\sigma(A) \subseteq W(A)^\circ$ and setting $\Omega = W(A)^\circ$, what are the
degree(s) of the associated extremal Blaschke product(s) $B_A$?
\end{question}

In Section \ref{sec:EFV}, we return to a general $A$ and $\Omega$ and study
the associated extremal functions and vectors. It is already known that
extremal functions  enjoy the following orthogonality property, see
\cite[Theorem~5.1]{CGL18}: if $f$ is  extremal  for $(A,\Omega)$,
if $\|f(A)\|>1$,  and if $x$ is a unit vector on which $f(A)$ attains its
norm (i.e. an extremal unit vector), then  $\langle f(A)x,x\rangle=0$.
We generalize this result in Section~\ref{sec:orthog}. In particular, if $x$
is an extremal unit vector, Theorem~\ref{T:orthog} shows that if
$f = f_1 \cdot f_2$ is a factorization of $f$ with each
$f_j \in H^\infty_1(\Omega)$, then
\[
\|f(A)\|^2 \langle f_1(A)x, x\rangle = \langle f_1(A)x, f_2(A)^{*} f_2(A) x\rangle.
\]
This can be viewed as a sort of cancellation theorem, particularly in the case
when $\|f(A)\| = 1$. In Theorem \ref{T:worthog}, we prove a similar result for
functions extremal with respect to the numerical radius.

In Section~\ref{sec:rep}, we provide representation theorems for extremal
vectors.
For example, using an extremal function $f$ for $(A, \Omega)$, we obtain
Theorem~\ref{T:rep}, which shows that for each associated extremal unit vector $x$, there is a
unique Borel probability measure $\mu$ defined on $\partial \Omega$ such that
for all $h \in \mathcal{A}(\Omega)$, we have
\[
\langle h(A)x, x\rangle = \int_{\partial \Omega} h \, d\mu.
\]
A similar result holds for vectors that are extremal with respect to the
numerical radius. To demonstrate Theorems \ref{T:orthog} and \ref{T:rep},
we apply  them to the extremal functions and vectors for
$(S_{\Theta}, \mathbb{D})$, see Examples \ref{ex:comp1} and \ref{ex:comp2}.
Furthermore, the connections between Crouzeix's conjecture and the structure
of extremal functions and vectors mentioned earlier motivate the following
question:
\begin{question} \label{q:cc}
Can Theorems \ref{T:orthog} or \ref{T:rep} be used to characterize extremal
functions and/or vectors associated to $(A, W(A)^\circ)$?
\end{question}
More open questions related to these topics are delineated throughout the
rest of the paper.


\section{Crouzeix's Conjecture via Pointwise Bounds and Condition Numbers}
\label{sec:CC}

In this section, $A$ is a contraction with distinct eigenvalues in $\mathbb{D}$. Note that $A$ factors as $X \Lambda X^{-1}$ with $\Lambda$ diagonal and, for  any $f$ defined on the eigenvalues of $A$, we have $f(A) = X f(\Lambda) X^{-1}$. To measure how separated the eigenvalues are, we define the \emph{pseudohyperbolic distance} between $z$ and $w$ in $\mathbb{D}$ as
\begin{equation}\label{eqn:pseudo}
\rho(z, w) := \left|\frac{z-w}{1 - \overline{w}z}\right|.
\end{equation}
We use classical function theory to study $\| f(A) \|$  for certain classes of functions $f$. Recall that
$H^2(\Omega)$ is the usual Hardy space on a domain $\Omega$, and let the algebra consisting of bounded analytic functions on $\Omega$ be denoted by $H^\infty(\Omega)$, with closed unit ball $H^{\infty}_1(\Omega)$. When $\Omega = \mathbb{D}$, we often simply write $H^2$ and $H^\infty$.

\subsection{Bounds via Interpolation Theory} \label{sec:cc1}

For a finite or infinite sequence $S = (z_j)$ of points in $\mathbb{D}$, we let
\begin{equation} \label{eqn:delta1} \delta_S = \inf_j  \prod_{k \ne j} \rho(z_j, z_k),\end{equation} where $\delta_S$ is called the \emph{separation constant} corresponding to $S$. The following result due to J. P. Earl connects this separation constant to interpolation problems:

\begin{theorem}[\cite{E70}]\label{Earl2} Let $S := (z_j)$ be a sequence in $\mathbb{D}$ with separation constant $\delta_S>0$ and let $(w_j)$ be a bounded sequence of complex numbers.  Then there exists $F \in H^\infty$ solving the interpolation problem $F(z_j) = w_j$, for $j = 1, 2, \ldots$,  with $\|F\|_{H^\infty} \le M(\delta_S)\sup_j|w_j|$, where
\[M(\delta) = \left( 1/\delta + \sqrt{1/\delta^2-1}\right)^2.\] 
\end{theorem}
Theorem \ref{Earl2} paired with von Neumann's inequality can be used to provide a bound on $\| f(A)\|$. In what follows, for an $n\times n$ matrix $A$ with distinct eigenvalues in $\mathbb{D}$, we define $\delta_{A}:=\delta_{\sigma(A)}$.

\begin{theorem}\label{crouzeix} Let $A$ be an $n \times n$ matrix with $\|A\| \le 1$ with distinct eigenvalues and suppose that $\sigma(A) \subset \mathbb{D}$.  If $f \in H^\infty$, then
\[\|f(A)\| \le M(\delta_{A}) \max_{z\in\sigma(A)} |f(z)|.\]
\end{theorem}

By rescaling, there is a version of Theorem \ref{crouzeix} for general matrices with distinct eigenvalues and $\sigma(A)\subset \{z\in\mathbb{C}\colon |z|<\|A\|\}$. Note however, that $\delta_{A}$ is not invariant under mappings $A\mapsto cA$ for $c>0$.

\begin{proof}
Write $\sigma(A) = \{z_1, \ldots, z_n\}$. By Theorem~\ref{Earl2} applied to $w_j := f(z_j)$, there exists $F \in H^\infty$ such that $F(z_j) = f(z_j)$ for all $j$ and $\|F\|_{H^\infty}\le M(\delta_{A}) \max_j |f(z_j)|$. Since $f = F$ on $\sigma(A)$ and the eigenvalues of $A$ are distinct, we have $f(A) = F(A)$. And, since $A$ is a contraction, von Neumann's inequality yields $\|F(A)\| \le \|F\|_{H^\infty}$. Putting this together we have
\[\|f(A)\| = \|F(A)\| \le \|F\|_{H^\infty} \le M(\delta_{A}) \max_{z\in\sigma(A)} |f(z)|,\]
the desired bound.
\end{proof}
In Theorem \ref{Earl2}, the $M(\delta)$
is a decreasing function of $\delta$ that tends to $1$ as $\delta \to 1^-$. 
Thus, in Theorem~\ref{crouzeix}, when the eigenvalues of $A$ are far apart, pseudohyperbolically speaking, the constant $M(\delta_{A})$ is close to $1$ and we need only consider the behavior of $f$ on the spectrum of $A$ to get an estimate on $\|f(A)\|$. 
Moreover, a computation shows that for $\delta \ge 2\sqrt{2}/3$, we have $M(\delta) \le 2$. Thus, for matrices with well-separated eigenvalues, the following strong form of Crouzeix's conjecture holds.

\begin{corollary}\label{cor:crouzeix} Let $A$ be an $n \times n$ matrix with $\| A \| \le 1$, $\sigma(A) = \{z_1, \ldots, z_n\}\subseteq \mathbb{D}$, and $\delta_{A} \ge 2\sqrt{2}/3$. Then for $f \in H^\infty$, we have
\[\|f(A)\| \le 2 \max_j |f(z_j)|.\]
\end{corollary}
 Note, however, that the assumption $\delta_A \geq 2 \sqrt{2} / 3 \approx 0.9428$ implies that the pseudohyperbolic
distance between each pair of eigenvalues is at least $(2 \sqrt{2} / 3 )^{1/(n-1)}$.  If the eigenvalues are uniformly
distributed around a circle of radius $r$ about the origin, for example, then when $n=2$, this means that $r$ must
be greater than about $0.707$; for $n=3$, $r > 0.861$; for $n=5$, $r > 0.942$; for $n=10$, $r > 0.981$;
for $n=100$, $r > 0.9995$, etc.  If $A$ 
 has an eigenvalue at the origin, then all other eigenvalues must have magnitude at least $(2 \sqrt{2} / 3 )^{1/(n-1)}$. 

The bound in Corollary \ref{cor:crouzeix} implies, under the assumptions, that the matrix $A$ is
near normal, in the sense of having a well-conditioned eigenvector matrix.
To see this, suppose that
\begin{equation}\label{eqn:normboundC}
\|p(A)\| \le C \max_{j} |p(z_j)| \quad \text{ for all } p \in \mathbb{C}[z].
\end{equation}
Write $p(A) = X p( \Lambda ) X^{-1}$ in the following form:
\[p(A) = \sum_{j = 1}^n p(z_j) x_j y_j^\ast,\] where $x_j$ is the $j$th column of $X$ and $y_j^\ast$ is the $j$th row of $X^{-1}$.
Now choose a polynomial $p_j$ such that $p_j(z_j) = 1$ and $p_j(z_k) = 0$ for $k \ne j$.
Applying inequality (\ref{eqn:normboundC}) to $p_j$, we see that $\| x_j y_j^{*} \| \leq C$,
whence 
\[ \| x_j \| \| y_j \|  = \|x_j\| \sup_{||z||=1} |y_j^\ast z| = \sup _{||z||=1} \| (x_j y_j^\ast)z \| = \| x_j y_j^\ast\| \leq C.\]
If each column of $X$ is taken to be of $2$-norm $1$, then each row of $X^{-1}$ has $2$-norm at most $C$. Therefore, the Frobenius norm of $X$ is at most $\sqrt{n}$ and the Frobenius norm of $X^{-1}$ is at most $\sqrt{n} C$. Since the operator norm of a matrix is less than or equal to the Frobenius norm, we have
\begin{equation}\eta(A) \le \kappa(X) = \|X\| \cdot \|X^{-1}\| \le n C,\label{kappabound1} \end{equation}
where $\eta(A)$ is the quantity defined in 
\eqref{eqn:k}.
Thus, under the assumptions of Corollary \ref{cor:crouzeix}, $\kappa (X) \leq 2n$. Actually, a somewhat stronger relation is known between the best-conditioned
eigenvector matrix in the operator norm and the best-conditioned eigenvector matrix in the 
Frobenius norm.  It is shown in \cite{RASmith} that
\[
n-2 + \kappa + \frac{1}{\kappa} \leq \kappa_F ,
\]
where $\kappa$ is the operator norm condition number and $\kappa_F$ is the Frobenius norm
condition number.  It follows that inequality (\ref{kappabound1}) can be replaced by
\[
\eta (A) \leq \frac{1}{2} \left( nC - n + 2 + \sqrt{( nC - n + 2 )^2 -4} \right) \leq nC - n + 2 ,
\]
and if $C=2$, then $\eta (A) \leq n+2$.

In fact a stronger bound on $\kappa (X)$ may be given when we interpolate with
Blaschke products instead of polynomials.
This estimate relates $\kappa(X)$ directly to the
separation constant $\delta_{A}$.
It requires a more general version of von Neumann's inequality for
holomorphic functions which follows from the same approximation
argument already used above.
The finite Blaschke products $h_j$ of degree $n-1$ defined by
\[h_j (z) := \frac{1}{\delta_j} \prod_{k \neq j}\frac{z - z_k}{1 - \overline{z}_k z} ,\quad
\delta_j := \prod_{k \neq j} \frac{z_j - z_k}{1 - \overline{z}_k z_j} ,\]
are the minimal norm interpolants that are $1$ at $z_j$ and $0$ at the other
eigenvalues of $A$. At the spectrum of $A$ they attain the same values
as the $p_j$, and hence $h_j (A) = p_j (A)$. Since $A$ is a contraction,
the generalized von Neumann's inequality yields
\[
\| p_j (A) \| = \| h_j (A) \| \leq \| h_j \|_{H^{\infty}} = 1/ | \delta_j |.
\]
Arguing the same way as above, we get that
the Frobenius norm of $X$ is at most $\sqrt{n}$ and that of $X^{-1}$ is at most
$\sqrt{ \sum_{j=1}^n 1/ | \delta_j |^2 }$;
thus the condition number of $X$ (in either the Frobenius norm or the operator norm) satisfies
\begin{equation}
\eta(A) \le \kappa (X) \leq \sqrt{n}~\sqrt{ \sum_j \frac{1}{| \delta_j |^2} }
\leq \frac{n}{\delta_{A}} .
\label{kappaofX}
\end{equation}
Using the result in \cite{RASmith}, we can subtract $n-2$ from the right-hand side of
(\ref{kappaofX}) to obtain a stronger bound on $\eta (A)$.

\subsection{Bounds via Condition Numbers}  \label{sec:cc2}

As before, let $A$ be an $n \times n$ matrix with distinct eigenvalues and decomposition $A = X \Lambda X^{-1}$. If the quantity $\eta(A)$  from \eqref{eqn:k} satisfies  $\eta(A) \le 2$, then Crouzeix's conjecture immediately holds for $A$. In general,  $\eta(A)$  can be arbitrarily large. However,  it is possible to obtain bounds on  $\eta(A)$ in the important case where $A$ is a matrix representation of a compressed shift $S_{\Theta}$. For additional background material concerning compressed shifts and their matrix representations, we refer the reader to \cite{GRSU} and Chapter $12$ in \cite{GMR}. 

To that end, let $\Theta$ be a finite Blaschke product with distinct zeros $z_1, \ldots, z_n \in \mathbb{D}$ and let  $b_{z_k}(z) = \frac{z - z_k}{1 - \overline{z}_kz}$ denote a single Blaschke factor.  A useful basis of $K_{\Theta}$ is the Takenaka-Malmquist basis\footnote{The name of this basis is not standard. According to \cite{GRSU} these appeared in Takenaka's 1925 paper, \cite{Ta25}. The text \cite{N86} discusses this basis for the case including infinite Blaschke products and uses the term ``Malmquist-Walsh'' basis.}, defined as follows
\[\varphi_1(z) := \frac{\sqrt{1 - |z_1|^2}}{1 - \overline{z_1}z},
~\mbox{and}~
\varphi_k(z) = \left(\prod_{j = 1}^{k-1} b_{z_j}\right) \frac{\sqrt{1 - |z_k|^2}}{1 - \overline{z_k}z}, \text{ for } k=2, \dots, n.\]
Writing $S_{\Theta}$ with respect to the Takenaka-Malmquist basis gives the matrix representation $M_{\Theta}$ where
\[[M_{\Theta}]_{i,j} = \left\{ \begin{array}{ll} z_i  & \text{ if } i=j \\
\prod_{k=i+1}^{j-1} (-\bar{z}_k) \sqrt{1-|z_i|^2} \sqrt{1-|z_j|^2} & \text{ if } i < j \\
0 & \text{ if } i >j \end{array} \right..\]

For example, if $\deg \Theta = 4$, then
\[ M_{\Theta} = \left[ \begin{array}{cccc} z_1 & \sqrt{1-|z_1|^2} \sqrt{1-|z_2|^2} & -\bar{z}_2 \sqrt{1-|z_1|^2} \sqrt{1-|z_3|^2} & \bar{z}_2 \bar{z_3}  \sqrt{1-|z_1|^2} \sqrt{1-|z_4|^2}  \\
0 & z_2 &  \sqrt{1-|z_2|^2} \sqrt{1-|z_3|^2} & -\bar{z}_3  \sqrt{1-|z_2|^2} \sqrt{1-|z_4|^2}  \\
0 & 0 & z_3 &  \sqrt{1-|z_3|^2} \sqrt{1-|z_4|^2} \\
0 & 0& 0& z_4 \end{array} \right].\]
Let $\Lambda$ be the diagonal matrix with diagonal entries $z_1, \dots, z_n$.  If $\Theta$ has distinct zeros, then $M_{\Theta}$ has distinct eigenvalues and so can be written as $X \Lambda X^{-1}$ for some matrix $X$.  Here are tractable formulas for $X$ and $X^{-1}$. The proof appears in the appendix in Section~\ref{sec:app}.

\begin{theorem} \label{thm:eigenvectors} Let $\Theta$ be a finite Blaschke product with distinct zeros $z_1, \dots, z_n$. Let $M_\Theta$ be the matrix representation of $S_\Theta$ with respect to the Takenaka-Malmquist basis. Then $M_{\Theta} = X \Lambda X^{-1}$, where  $\Lambda$ is the diagonal matrix with $\Lambda_{ii} = z_i$ for $1 \le i \le n$ and the entries of $X$ and $X^{-1}$ are given by
\[
\begin{aligned} X_{ij} &= \left\{ \begin{array}{ll} 1 & \text{ if } i=j \\
 \frac{ \sqrt{1-|z_i|^2} \sqrt{1-|z_j|^2}}{z_j - z_i}\prod_{k=i+1}^{j-1} \left( \frac{1-\bar{z}_k z_j}{z_j - z_k}\right) & \text{ if } i < j \\
0 & \text{ if } i >j \end{array} \right. \\
\\
X^{-1}_{ij} &= \left\{ \begin{array}{ll} 1 & \text{ if } i=j \\
 \frac{ \sqrt{1-|z_i|^2} \sqrt{1-|z_j|^2}}{z_i - z_j}\prod_{k=i+1}^{j-1} \left( \frac{1-\bar{z}_k z_i}{z_i- z_k}\right) & \text{ if } i < j \\
0 & \text{ if } i >j \end{array} \right.
\end{aligned}\]
\end{theorem}

\begin{remark} This theorem gives an initial bound on the condition number of $X$, namely
 \[\kappa( X ) \le \frac{n(n+1)}{2} ( \max_{i,j} | X_{ij} | )
( \max_{i,j} | X^{-1}_{ij} | ).\]
If the zeros of $\Theta$  are sufficiently separated (in the Euclidean and pseudohyperbolic metrics) and at least $n-1$ are sufficiently near the unit circle $\mathbb{T}$, then the formulas in Theorem \ref{thm:eigenvectors} show that we can make the off-diagonal entries of $X$ and $X^{-1}$ arbitrarily close to $0$ and hence the
condition number of $X$ arbitrarily close to $1$.
 \end{remark}

Using the formulas for $X$ and $X^{-1}$, we can also obtain the following bound on the condition number of $X$. Note that this also provides a bound on $\eta(M_\Theta)$
for the matrix representation $M_\Theta$ of the compressed shift $S_\Theta$.

 \begin{theorem} \label{thm:condition} If the eigenvector matrix $X$ is given
as in Theorem~\ref{thm:eigenvectors},
then
\[
\eta(M_\Theta) \le \kappa(X) \le  \frac{8}{\delta_{\Theta}^6}
\Big( 1 - 2\log \delta_{\Theta}\Big),
\]
 where $\eta(M_{\Theta})$ denotes the minimal condition number of $M_{\Theta}$ defined in \eqref{eqn:k}  and 
 $\delta_{\Theta}$ denotes the separation constant of the
zeros of $\Theta$ defined in \eqref{eqn:delta1}.
\end{theorem}

As the bound in Theorem \ref{thm:condition} does not depend on $n$, it seems better than the bound in
\eqref{kappaofX}
in situations where $n$ is large.  Nevertheless, it includes significant dependence on $\delta_{\Theta}$, and the appearance of the constant $8$ prevents this estimate from being sharp as $\delta_\Theta\to 1$.  These concerns motivate Question~\ref{q:condition}, which was posed in the introduction.

To prove Theorem \ref{thm:condition}, we need the following lemma, which is likely well known.

\begin{lemma} \label{lem:M} Let $\Theta$ be a finite Blaschke product with distinct zeros $z_1, \dots, z_n$,  and let $\delta_{\Theta}$  denote the separation constant of the
zeros of $\Theta$ given in \eqref{eqn:delta1}. Define $g_\ell(z) = \frac{\sqrt{1-|z_\ell|^2}}{1-\bar{z}_\ell z}$ for $\ell=1,\dots, n$ and let $G$ be the $n\times n$ Gramian matrix defined by $G_{ij} = \langle g_i, g_j \rangle_{H^2}$. Then  \[ \| G \|^2 \le \frac{2}{\delta_\Theta^4}\Big(1-2 \log \delta_\Theta \Big).\]
\end{lemma}

\begin{proof} By Lemma $3$ in \cite{SS}, if $(w_k)$ is a square summable sequence, then there is a $g \in H^2(\mathbb{D})$ such that
\[ \| g \|_{H^{2}}^2 \le \frac{2}{\delta_\Theta^4}\big( 1- 2 \log \delta_\Theta\big) \sum_{k=1}^{\infty} |w_k|^2 \ \text{ and } \ g(z_k) (1-|z_k|^2)^{1/2}= w_k \text{ for } k=1,2,\dots.\]
By Lemma $1$ in \cite{SS}, we can conclude that
\[ \sum_{k=1}^{\infty} |g(z_k)|^2 (1-|z_k|^2) \le \frac{2}{\delta_\Theta^4}\Big(1-2 \log \delta_\Theta\Big) \| g\|^2_{H^2} \text{ for all } g\in H^2(\mathbb{D}).\]
Then Proposition $9.5$ in \cite{AM02} implies that $\| G\|^2 \le \frac{2}{\delta_\Theta^4}\Big(1-2 \log \delta_\Theta\Big)$, which completes the proof.
\end{proof}

\begin{proof}[Proof of Theorem \ref{thm:condition}] We use the estimate $\| X\| \le 2 w(X),$ where $w(X)$ denotes the numerical radius. Then fixing $y \in \mathbb{C}^n$ with $\| y \|=1$, we have
\begin{equation} \label{eqn:ws} | \langle X y , y \rangle |= \left| \sum_{j=1}^n \sum_{i <j} \frac{\sqrt{1-|z_i|^2}\sqrt{1-|z_j|^2}}{1-\bar{z}_i z_j} \prod_{k=i}^{j-1} \left( \frac{ 1-\bar{z}_k z_j}{z_j-z_k} \right)\bar{y}_i y_j + \sum_{j=1}^n | y_j|^2 \right|.\end{equation}
For any $\ell =1, \dots, n$, define the functions
\[ g_\ell(z) = \frac{\sqrt{1-|z_\ell|^2}}{1-\bar{z}_\ell z}, \quad B_{\ell}(z) = \prod_{k \ne \ell} \frac{z-z_k}{1-\bar{z}_kz}, \quad C^+_{\ell}(z) = \prod_{k>\ell} \frac{z-z_k}{1-\bar{z}_kz},  \quad  D^-_{\ell}(z) = \prod_{k<\ell} \frac{z-z_k}{1-\bar{z}_kz}.\]
Then \eqref{eqn:ws} can be rewritten as:
\[
| \langle X y , y \rangle | = \left |  \sum_{j=1}^n \sum_{i <j} \langle g_i, g_j \rangle \frac{ C_j^+(z_j) D^-_i(z_j)}{B_j(z_j)} \bar{y}_i y_j + \sum_{j=1}^n | y_j|^2 \right| \\
= \left |  \sum_{i,j=1}^n\langle g_i, g_j \rangle \frac{ C_j^+(z_j) D^-_i(z_j)}{B_j(z_j)} \bar{y}_i y_j  \right|,
\]
where we used the fact that if $i=j$, then $\langle g_i, g_j \rangle \frac{ C_j^+(z_j) D^-_i(z_j)}{B_j(z_j)}=1$ and if $i>j$, then $D^-_i(z_j)=0$. Furthermore, observe that each $\langle g_k D^-_k, g_{\ell} D^-_{\ell} \rangle = \delta_{k\ell}.$ This and Lemma \ref{lem:M} give:
\[
\begin{aligned}
| \langle X y, y \rangle | &= \left|
 \left \langle  \sum_{i=1}^n g_i D^-_i \bar{y}_i,  \sum_{j=1}^n g_j  \frac{ \overline{C_j^+(z_j)}}{\overline{B_j(z_j)}}
 \bar{y}_j \right \rangle \right | \\
 &\le \left \|  \sum_{i=1}^n g_i D^-_i\bar{y}_i \right \| \cdot \left \| \sum_{j=1}^n g_j  \frac{ \overline{C_j^+(z_j)}}{\overline{B_j(z_j)}}
 \bar{y}_j   \right \| \\
 & \le \frac{1}{\delta_{\Theta}} \left(  \sum_{i=1}^n  \| g_i D^-_i \bar{y}_i\|^2 \right )^{1/2}   \left \|  \sum_{j=1}^n g_j
 \bar{y}_j   \right \|\\
&  \le  \frac{\sqrt{2}}{\delta_\Theta^3}\sqrt{1-2 \log \delta_\Theta} \| y \|^2 =  \frac{\sqrt{2}}{\delta_\Theta^3}\sqrt{1-2 \log \delta_\Theta}.
\end{aligned}
\]
This shows that $\| X \| \le 2  \frac{\sqrt{2}}{\delta_\Theta^3}\sqrt{1-2 \log \delta_\Theta},$ and an analogous argument gives the same bound for $\|X^{-1}\|$.
\end{proof}

\section{Examples of Extremal Functions and Vectors} \label{sec:extremal}
Let $\Omega$ be a bounded simply connected domain with smooth boundary containing the spectrum of an $n \times n$ matrix $A$. We consider functions $f \in H_1^\infty(\Omega)$, the closed unit ball in $H^{\infty}(\Omega)$, for which $\sup_f\|f(A)\|$, taken over all $f\in H_1^\infty(\Omega)$, is attained. Recall that such a function is called \emph{extremal} for $(A,\Omega)$ and if, furthermore,  $x$ is a non-zero vector where $\| f(A) x \| = \| f(A) \| \| x\| $, then $x$ is called an associated \emph{extremal} vector. As discussed in the introduction, the study of such functions is closely related to recent proofs and investigations of Crouzeix's conjecture.

One can also measure the size of $f(A)$ via its numerical radius. Given $(A, \Omega)$, we say that $f$ is \emph{$w$-extremal}, if  $f \in H_1^\infty(\Omega)$ is a function for which $\sup_f w(f(A))$, taken over all $f\in H_1^\infty(\Omega)$, is attained. A vector $y$ is an associated \emph{$w$-extremal} vector if $| \langle f(A) y/\| y\| ,y/\|y\|\rangle | =  w(f(A)).$

In this section, we consider two classes of examples of extremal functions and vectors.

\subsection{Compressions of the Shift $S_{\Theta}$ with $\Omega = \mathbb{D}$} \label{sec:comp}

Let $\Theta$ be a finite Blaschke product. Recall that $K_\Theta = H^2 \ominus \Theta H^2$ and the compression of the shift with symbol $\Theta$ is defined by $S_\Theta = P_\Theta S|_{K_\Theta}$, where $P_\Theta$ is the orthogonal projection of $H^2$ onto $K_\Theta$ and $S$ is the shift operator.  In \cite[Theorem 2, pp. 22]{GR09}, Garcia and Ross showed that the extremal functions for  $(S_\Theta, \mathbb{D})$ are exactly the finite Blaschke products $B$ with $\deg B < \deg \Theta$. We encode their result in the following theorem:

\begin{theorem}\label{thm:norm} Let $\Theta$ be a finite Blaschke product with $\deg \Theta =n$ and let  $f \in H^{\infty}_1(\mathbb{D})$. Then $\| f(S_{\Theta}) \| \le 1$. Moreover $\| f(S_{\Theta}) \| =1$ if and only if $f$ is a finite Blaschke product with $\deg f < \deg \Theta.$ \end{theorem}

Here we present a simple proof of this result when $\Theta$ has distinct zeros $z_1, \dots, z_n \in \mathbb{D}$.

\begin{proof} First, fix $f \in H^{\infty}_1(\mathbb{D})$. Then von Neumann's inequality implies that $\| f(S_{\Theta})\| \le 1.$

Now fix $f \in H^{\infty}_1(\mathbb{D})$ with $\| f(S_{\Theta})\| =1$.  Then by \cite[Proposition 5.1]{Sar67}, there is a unique $\psi \in H^\infty$ such that $\|\psi\|_\infty = \|f(S_\Theta)\|$ and $f(S_\Theta) = P_\Theta  T_\psi|_{K_\Theta}$,  where $T_\psi$ denotes multiplication by $\psi$. Moreover, $\psi$ is a finite Blaschke product of degree at most $n-1$.  Sarason's work \cite[pp. 187]{Sar67} implies that $\psi(z_j) = f(z_j)$ for $j = 1, \ldots, n$.  Since the $\deg \psi <n$,  the interpolation problem has a unique solution in $H^\infty_1(\mathbb{D})$ (see \cite[pp. 77, Lemma 6.19]{AM02}) and so, $f=\psi.$

Similarly, if we begin with a Blaschke product $f$ of $\deg f <n$, then \cite[Proposition 5.1]{Sar67} again implies that the existence of a unique $\psi$ with $\|\psi\|_\infty = \|f(S_\Theta)\| \le 1$. Again, the associated interpolation problem has a unique solution in $H^{\infty}_1(\mathbb{D})$ and so, we can conclude that $\psi =f $ and $\|f(S_\Theta)\| = \|\psi\|_\infty = 1,$ which completes the proof.  \end{proof}

Since $K_{\Theta}$ is finite dimensional, for $B$ a finite Blaschke product with $\deg B < \deg \Theta$, there is some nonzero vector $x \in K_{\Theta}$ such that $\| B(S_{\Theta}) x \| =\| x\|.$ Note that if $p$ is a polynomial, then $p(S_\Theta) = P_\Theta T_p |_{K_\Theta}$. This can then be extended to all functions in the disk algebra, $\mathcal{A}(\mathbb{D})$. The extremal vectors (which Sarason calls {\it maximal vectors} in \cite{Sar67}) can be characterized as follows.

\begin{theorem} \label{thm:vector} Let $\Theta$ be a finite Blaschke product with zeros $z_1, \dots, z_n$ and $B$ be a finite Blaschke product with zeros $a_1, \dots, a_J$. Assume $J <n.$  Then for each $x \in K_{\Theta}$, the following are equivalent:
\begin{enumerate}
\item\label{item:1} $\| B(S_\Theta) x \| = \| x\|$;
\item\label{item:2} $Bx \in K_{\Theta}$;
\item\label{item:3} $\displaystyle x(z) = \dfrac{p(z)\prod_{j=1}^J(1-\bar{a}_j z)}{\prod_{i=1}^n (1-\bar{\lam}_iz)}$ for some polynomial $p$ with $\deg p < n-J$;
\end{enumerate}

\end{theorem}

\begin{proof} For each $f\in \mathcal{A}(\mathbb{D})$, we know that $f(S_{\Theta}) = P_{\Theta} T_f |_{K_{\Theta}}$, where $T_f$ denotes multiplication by $f$. Applying this to the Blaschke product $B$ and $x \in K_{\Theta}$ gives
 $\| B(S_{\Theta}) x \| = \| P_{\Theta} (Bx) \|$. Observe that
\[ \| x\|^2_{K_{\Theta}} = \| Bx \|_{H^2}^2 = \| P_{\Theta}(Bx) \|_{H^2}^2 + \| (I - P_{\Theta}) (Bx) \|_{H^2}^2.\]
Thus, $\| B(S_{\Theta}) x \|  = \| x\|$ if and only if $(I - P_{\Theta}) (Bx)  =0$, or equivalently $Bx \in K_{\Theta}$.  Therefore \eqref{item:1} holds if and only if \eqref{item:2} holds.

Each  $ x\in K_{\Theta}$ is exactly of the form $\frac{q(z)}{\prod_{i=1}^n (1-\bar{\lam}_iz)}$ for some polynomial $q$ with $\deg q <n$. Thus, for $x \in K_{\Theta}$, the function $Bx \in K_{\Theta}$ if and only if $q(z) =  p(z)\prod_{j=1}^J(1-\bar{a}_j z)$ for some polynomial $p$ with $\deg p < n-J.$  Consequently, \eqref{item:2} holds if and only if \eqref{item:3} holds.
\end{proof}

Note that extremal vectors can be used to build new bases of $K_{\Theta}$.

\begin{remark} Let $\Theta$ and $B$ be finite Blaschke products with $\deg \Theta =n$ and $\deg B = n-1$. Factor $B = B_1 \cdots B_{n-1}$ into its component Blaschke factors. Let $ x\in K_{\Theta}$ be an extremal vector of $S_{\Theta}$ associated to the extremal function $B.$ Then $Bx \in K_{\Theta}$. Since $K_{\Theta} = H^2 \cap \overline{ \Theta z H^2}$, writing a factorization of $B = C C^\prime$ and multiplying by $\overline{C^\prime}$ shows that $B_1 x, B_1B_2x, \dots, B x \in K_{\Theta}$. Then linear independence  implies that the set $\{ B_1x, B_1 B_2 x, \dots, B x \}$ is a basis of $K_{\Theta}$. \end{remark}

In contrast to these norm results, obtaining a general characterization for $w$-extremal functions or vectors in this setting seems quite difficult, prompting the question:

\begin{question} What are the $w$-extremal functions and vectors for  $(S_{\Theta}, \mathbb{D})$? \end{question}

Recent work by Gaaya and Gorkin-Partington has revealed precise formulas for some $w(S_{\Theta})$, see \cite{GA13, GP19}. This suggests that it might be possible to answer parts of this question for very specialized $\Theta.$

\subsection{General $n \times n$ matrix $A$ with $\Omega = W(A)^\circ$} \label{sec:degree}

As mentioned earlier, an extremal (or $w$-extremal) function $f$ for $(A,\Omega)$ has the form $f=B_A\circ\phi$,
where $\phi$ is a bijective conformal map of $\Omega$ onto $\mathbb{D}$ and $B_A$ is a Blaschke product of degree at most $n-1$. In this section we study the basic structure of such extremal functions. In general this is a very challenging question, but we do make partial progress on the following question:

\begin{question} Given $A$ with $\sigma(A) \subseteq W(A)^\circ$ and $\Omega=W(A)^\circ$, what is the maximum degree of a Blaschke product $B_A$ corresponding to an extremal function $f$? \end{question}

We say that $f$ is of \emph{maximal degree} if $\deg B_A=n-1$.  
Some numerical computations suggest that extremal functions $f$ for randomly
generated matrices $A$ often have less than maximal degree.  For example, Figure \ref{fig:extremalfs}
shows the degrees of identified extremal functions $f$ for $500$ random dense complex matrices
of size $3 \times 3$, $4 \times 4$, and $5 \times 5$.  The extremal functions were computed
using a conformal mapping routine in the {\em chebfun} package, see https://www.chebfun.org/, to produce the
mapping $\phi$ from $W(A)$ to $\mathbb{D}$ and then an optimization code to find the roots $\alpha_j$,
$j=1, \ldots , n-1$ of a Blaschke product $B$ that maximized $\| B( \phi (A) ) \|$.  The $\alpha_j$s
were constrained to have magnitude at most $1$, and if the code returned some $\alpha_j$s with
magnitude extremely close to $1$, then we determined that the actual degree of $B$ was less than
maximal, since if $| \alpha | = 1$ then the Blaschke factor $(z - \alpha )/(1 - \bar{\alpha} z )$
is just a scalar of modulus 1.  There is no guarantee that the code has found the true extremal function $f$,
but we tested several of the known properties of extremal functions (such as the orthogonality condition
$\langle f(A)x,x \rangle = 0$), and these all held to a close approximation in the results that we
recorded.  Also, one cannot use numerical results to {\em definitively} determine the degree of $f$;
it could be that a Blaschke product has maximal degree but has some roots with magnitude
extremely close to $1$; if a numerically computed $\alpha_j$ had magnitude greater than $0.9999$,
we concluded that $| \alpha_j |$ was actually $1$, so the Blaschke factor did not
add to the degree.  Note, from Figure \ref{fig:extremalfs}, that for $3 \times 3$ random matrices,
most extremal $f$'s had maximal degree $2$; for $4 \times 4$ random matrices the extremal $f$ was less likely
to have the maximal degree $3$; and for $5 \times 5$ random matrices, only one of the $500$ examples tested
had an extremal function $f$ with maximal degree $4$.

\begin{figure}[ht]
\centerline{\epsfig{file=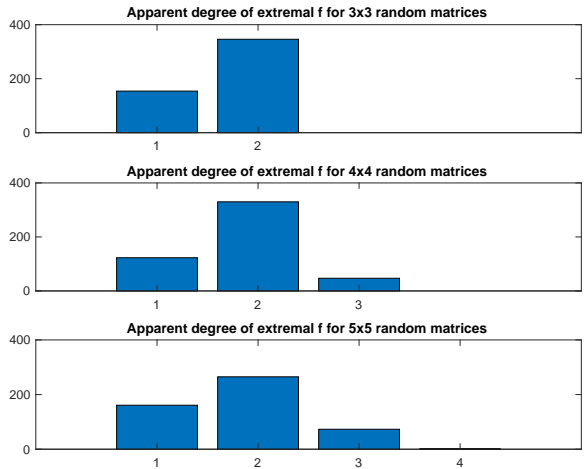,width=3in}}
\caption{Apparent degrees of extremal functions for random dense complex matrices of size $3 \times 3$,
$4 \times 4$, and $5 \times 5$.}
\label{fig:extremalfs}
\end{figure}

\begin{remark}
In a few special cases, it has been proved that the extremal function $f$ has
less than maximal degree. For instance, Li \cite{KenanLi} has shown that matrices of the form
\[
\left[ \begin{array}{ccc} 0 & 1 & 0 \\ 0 & 0 & 1-t \\ 0 & 0 & 0 \end{array} \right] ,~~ t \in ( 1 -
1/ \sqrt{3} , \sqrt{3} - 1 ] ,
\]
have $B_1 (z) = z$ as the unique extremal Blaschke product, while for $t \in [0, 1 - 1/ \sqrt{3} )$
the unique extremal Blaschke product is $B_2 (z) = z^2$.  At the point $t = 1 - 1/ \sqrt{3}$, both
Blaschke products give the same value for $\| B \circ \phi (A) \|$, where $\phi (z) = z/r$,
$r = \sqrt{1 + (1-t )^2}/2$ is the conformal mapping from $W(A)$ to $\mathbb{D}$.
\end{remark}

\begin{remark}
There are also some cases where the extremal Blaschke product is
known to be of maximal degree.
In \cite[Theorem 2.5]{Cr04}, Crouzeix considered the setting of $2 \times2$ matrices. For a fixed $2\times 2$ matrix $A$, he showed that the constant in \eqref{eqn:cc} is some number $\psi(A)$ so that $\psi(A)>1$ when $A$ is non-normal. This shows that as long as $A$ is not normal, $B_A$ cannot be a constant. Thus, it has to have degree $1$ and so, for $2\times 2$ matrices, any extremal $f$ of the form $B_A \circ \phi$ must have maximal degree.
\end{remark}

We now show that, for general $n$, there is an open set of $n\times n$ matrices whose extremal functions have maximal degree.

\begin{theorem}\label{T:openset}
For each $n\ge2$, there exists a non-empty open set $U$ of $n\times n$ matrices such that, for each $A\in U$, we have $\sigma(A)\subset W(A)^\circ$
and all extremal functions for $(A,W(A)^\circ)$ are of maximal degree.
\end{theorem}

We first consider an example, which may be well known. This will then be used in the proof of Theorem~\ref{T:openset}.

\begin{theorem}\label{T:Crabb}
Let $n\ge 2$, and let $C$ be the $n\times n$ matrix

\begin{equation}\label{E:Crabb}
C:=\begin{pmatrix}
0 &0 &\dots &0 &0 &0\\
\sqrt2 &0 &\dots &0 &0 &0\\
0 &1 &\dots &0 &0 &0\\
\vdots &\vdots&\ddots &\vdots  &\vdots &\vdots\\
0 &0  &\hdots &1 &0 &0\\
0 &0  &\hdots &0 &\sqrt2 &0
\end{pmatrix}.
\end{equation}
Then $\sigma(C)=\{0\}$ and $W(C)=\overline{\mathbb{D}}$.
The only extremal functions for $(C,\mathbb{D})$ are of the form $f(z)=\gamma z^{n-1}$, where $\gamma\in\mathbb{T}$.
\end{theorem}

\begin{proof}
It is obvious that $\sigma(C)=\{0\}$. Also, it is well known that $W(C)=\overline{\mathbb{D}}$, see \cite{Ch13}.
Let $f$ be an extremal function for $(C,\mathbb{D})$.
We have $C:=D^{-1}JD$, where $D,J$ are the $n\times n$  diagonal and Jordan matrices given respectively by
\[ D:=\begin{pmatrix}
\sqrt2 &0 &\dots &0 &0\\
0 &1 &\dots &0 &0\\
\vdots &\vdots&\ddots &\vdots  &\vdots\\
0 &0  &\hdots &1 &0\\
0 &0 &\hdots &0 &\frac{1}{\sqrt2}
\end{pmatrix}
\quad\text{and}\quad
J:=\begin{pmatrix}
0 &0 &\dots &0 &0\\
1 &0 &\dots &0 &0\\
0 &1 &\dots &0 &0\\
\vdots &\vdots&\ddots &\vdots  &\vdots\\
0 &0  &\hdots &1 &0
\end{pmatrix}.\]
Hence, writing $e_1,\dots,e_n$ for the standard unit vector basis of $\mathbb{C}^n$,
we have
\[
C^{n-1}e_1=D^{-1}J^{n-1}De_1=D^{-1}J^{n-1}\sqrt{2}e_1=2D^{-1}e_n=2e_n,\]
so $\|C^{n-1}\|\ge2$.
Let $g(z) = z^{n-1}$. As  $f$ is extremal, it follows that
\[
\|f(C)\|\ge \|g(C)\| \ge 2.
\]
On the other hand, since $J$ is a contraction, von Neumann's inequality implies that $\|f(J)\|\le1$,
and as $\|D\|=\|D^{-1}\|=\sqrt2$, we get
\[
\|f(C)\|=\|D^{-1}f(J)D\|\le \|D^{-1}\|\|f(J)\|\|D\|\le 2.
\]
Thus, if $x$ is a unit vector of $\mathbb{C}^n$ on which $f(C)$ attains its norm, then we must have
\[
\|Dx\|=\sqrt{2}, \quad \|f(J)Dx\|=\sqrt2, \quad\text{and}\quad \|D^{-1}f(J)Dx\|=2.
\]
The first of these equalities implies that $x$ is a multiple of $e_1$,
and the second and third together imply that $f(J)Dx$ is a multiple of $e_n$.
It follows that  $f(J)e_1=\gamma e_n$, where necessarily $|\gamma|=1$.

Now let $f(z)=\sum_{k\ge0}\beta_kz^k$ be the Taylor expansion of $f$ around zero.
Then $f(J)e_1=\sum_{k=0}^{n-1}\beta_ke_{k+1}$. In particular, we must have
$\beta_{n-1}=\gamma$. But also, by Parseval's theorem, $\sum_{k\ge0}|\beta_k|^2=\|f\|_{L^2(\mathbb{T})}^2\le1$.
Therefore $\beta_k=0$ for all $k\ne n-1$. Thus $f(z)=\gamma z^{n-1}$.
\end{proof}

\begin{remark} The matrix $C$ that we consider here appears in other work, including that of Crabb \cite{Crabb}, Choi \cite{Ch13}, and Greenbaum and Overton \cite{GO}.\end{remark}

We now return to Theorem~\ref{T:openset}.

\begin{proof}[Proof of Theorem~\ref{T:openset}]
Let $C$ be the $n\times n$ matrix defined in the previous theorem. We claim that there is an open neighborhood $U$ of $C$ in $M_n(\mathbb{C})$ such that, for each $A\in U$, we have $\sigma(A)\subset W(A)^\circ$ and every extremal $f$ for $(A,W(A)^\circ)$ is of maximal degree. We shall prove this by contradiction.

First of all, it is easy to see that there exists $\delta>0$ such that
\[
\|A-C\|<\delta\quad\Rightarrow \quad\sigma(A)\subset\{|z|<1/2\}\subset W(A)^\circ.
\]
Thus, if the claim is false, then there exists a sequence $(A_k)$ of $n\times n$ matrices converging to $C$
such that, for each $k$, there is an $f_k$ extremal for $(A_k,W(A_k)^\circ)$ which is not of maximal degree.
Replacing $A_k$ by $r_kA_k$, where $(r_k)$ is a suitable positive sequence tending to $1$,
we may further suppose that $W(A_k)\subset \mathbb{D}$ for all $k$. Replacing $(A_k)$ by a subsequence, if necessary, we can suppose that $f_k\to f$ locally uniformly on $\mathbb{D}$,
where  $f$ is bounded by $1$ on $\mathbb{D}$.
Given a holomorphic function $g:\mathbb{D}\to\overline{\mathbb{D}}$, extremality of $f_k$ for $(A_k,W(A_k)^\circ)$
implies that $\|g(A_k)\|\le\|f_k(A_k)\|$.
Since the spectra of $A_k$ and $C$ remain inside a fixed compact subset of $\mathbb{D}$,
we also have $g(A_k)\to g(C)$ and $f_k(A_k)\to f(C)$ as $k\to\infty$. It follows that $\|g(C)\|\le \|f(C)\|$.
Therefore $f$ is extremal for $(C,\mathbb{D})$. By Theorem~\ref{T:Crabb}, it must be of the form $f(z)=\gamma z^{n-1}$ for some $\gamma\in\mathbb{T}$. In particular, it has a zero of order $n-1$ at the origin. By Hurwitz's theorem, for all sufficiently large $k$, the function $f_k$ must have at least $n-1$ zeros. This contradicts the fact $f_k$ is not of maximal degree.
\end{proof}

As a $w$-extremal function is also of the form $\tilde{B}_A \circ \phi$ for some Blaschke product $\tilde{B}_A$ with degree strictly less than $n$, it makes sense to ask:

\begin{question} For an $n\times n$ matrix $A$ with $\sigma(A) \subset W(A)^\circ$ and $\Omega =W(A)^\circ$, what is the relationship between the extremal $B_A$ and the $w$-extremal $\tilde{B}_A$?
\end{question}
 This is an interesting question, because there is a close connection between the extremal and $w$-extremal problems.
It was shown in \cite{BCK18} that  inequality \eqref{eqn:cc}  holds
if and only if
\begin{equation}\label{eqn:w}
w(p(A)) \le \widetilde{C} \sup_{z \in W(A)} | p(z) | \quad \text{ for all } p \in \mathbb{C}[z],
\end{equation}
where $\widetilde{C}:=(C+C^{-1})/2$. In particular, \eqref{eqn:cc} holds with $C=2$ if and only if
\eqref{eqn:w} holds with $\widetilde{C}=5/4$.
\section{Structure of Extremal Functions and Vectors}\label{sec:EFV}

As in Section \ref{sec:extremal}, we let $\Omega$ be a bounded simply connected domain with smooth boundary containing the spectrum of an $n \times n$ matrix $A$. We study the structure of both (1) extremal functions and their associated extremal vectors and (2) $w$-extremal functions and their associated $w$-extremal vectors for $(A, \Omega).$ As noted in Question \ref{q:cc}, the hope is that these structural results will aid in characterizing the extremal functions (and vectors) that have played a vital role in recent investigations of the Crouzeix conjecture.

\subsection{Orthogonality Properties for Extremal Functions} \label{sec:orthog}
First recall that extremal functions  enjoy the following orthogonality property \cite[Theorem~5.1]{CGL18}:
if $f$ is  extremal  for $(A,\Omega)$ with  $\|f(A)\|>1$ and $x$ is an associated extremal vector, then
 $\langle f(A)x,x\rangle=0$. The authors of \cite{CGL18} used this result to give a new proof of the theorem of Okubo and Ando that, if  $W(A)^\circ \subset \mathbb{D}$  and $f$ is holomorphic on~$\mathbb{D}$, then $\|f(A)\|\le 2\sup_{z\in \mathbb{D}}|f(z)|$. We generalize the orthogonality property as follows.

\begin{theorem}\label{T:orthog}
Let $f$ be extremal  for $(A,\Omega)$, and
let $x$ be an associated extremal unit vector.
Let $f=f_1f_2$ be a factorization of $f$ as a product of functions $f_1, f_2\in H_1^\infty(\Omega)$.
Then
\begin{equation}\label{E:orthog}
\Bigl\langle f_1(A)x,\Bigl(\|f(A)\|^2I-f_2(A)^*f_2(A)\Bigr)x\Bigr\rangle=0.
\end{equation}
\end{theorem}

\begin{remark}
In particular, taking $f_1:=f$ and $f_2:=1$, the relation \eqref{E:orthog} reduces to
\begin{equation}\label{E:orthog2}
(\|f(A)\|^2-1)\langle f(A)x,x\rangle=0.
\end{equation}
If further $\|f(A)\|>1$, then  $\langle f(A)x,x\rangle=0$, and we recapture the result mentioned earlier.
\end{remark}

\begin{proof}[Proof of Theorem~\ref{T:orthog}]
For $a\in\mathbb{D}$, let $\phi_a(z):=(z-a)/(1-\overline{a}z)$. Then $\phi_a$ maps $\mathbb{D}$ into $\mathbb{D}$ and
\[
\phi_a(z)=z-a+\overline{a}z^2+O(|a|^2)
\quad(a\to0).
\]
As $f$ is extremal for $(A,\Omega)$, we have $\|(\phi_a(f_1(A))f_2(A)\|\le\|f(A)\|$, so it follows that
\begin{align*}
\|f(A)\|^2&\ge
\Re \langle \phi_a(f_1(A))f_2(A)x,f(A)x\rangle\\
&=\Re \langle (f_1(A)-aI+\overline{a}f_1(A)^2+O(|a|^2))f_2(A)x,f(A)x\rangle\\
&=\|f(A)x\|^2-\Re\bigl(a\langle f_2(A)x,f(A)x\rangle\bigr)+ \Re\bigl(\overline{a}\langle f_1(A)^2f_2(A)x,f(A)x\rangle\bigr)+O(|a|^2)\\
&=\|f(A)x\|^2-\Re\bigl(a\langle f_2(A)^*f_2(A)x,f_1(A)x\rangle\bigr)+ \Re\bigl(\overline{a}\langle f_1(A)x,f(A)^*f(A)x\rangle\bigr)+O(|a|^2)\\
&=\|f(A)\|^2+\Re\bigl(\overline{a}\langle f_1(A)x,(f(A)^*f(A)-f_2(A)^*f_2(A))x\rangle\bigr)+O(|a|^2).
\end{align*}
Letting $a\to0$, and noting that the argument of $a$ is arbitrary, it follows that
\[
\langle f_1(A)x,(f(A)^*f(A)-f_2(A)^*f_2(A))x\rangle=0.
\]
Finally, since $f(A)$ attains its norm at $x$, we have $f(A)^*f(A)x=\|f(A)\|^2x$, which gives \eqref{E:orthog}.
\end{proof}

\begin{example} \label{ex:comp1} Let $\Theta$ be a finite Blaschke product with $\deg \Theta =n$, and set $A = S_{\Theta}$ and $\Omega = \mathbb{D}$. By Theorem \ref{thm:norm}, the extremal $f$ are exactly the finite Blaschke products $B$ with $\deg B <n,$ and by Theorem \ref{thm:vector}, the associated extremal unit vectors are exactly the unit vectors $x \in K_{\Theta}$ with $Bx \in K_{\Theta}$. Then Theorem \ref{T:orthog} and  $ \|B(S_\Theta) \| =1$ imply that for every factorization of $B = B_1 B_2$, we have
\[\left  \langle B_1(S_\Theta)x, \left(I - B_2(S_\Theta)^\ast B_2(S_\Theta) \right)x \right \rangle = 0.\]
\end{example}

There is an analogous result for $w$-extremal functions $f$ for $(A,\Omega)$.

\begin{theorem}\label{T:worthog}
Let $f$ be $w$-extremal for $(A,\Omega)$, and let $y$ be an associated $w$-extremal unit vector. Let $f=f_1f_2$ be a factorization of $f$ as a product of functions $f_1, f_2\in H_1^\infty(\Omega).$
Then
\begin{equation}\label{E:worthog}
\langle f_1(A)f(A)y,y\rangle=
\langle y,f_2(A)y\rangle.
\end{equation}
In particular, taking $f_1:=f$ and $f_2:=1$, we have
\[
\langle f(A)^2y,y\rangle=1.
\]
\end{theorem}

\begin{proof}
Define $\phi_a$ as before.
As $f$ is $w$-extremal for $(A,\Omega)$, we have
$w(\phi_a(f_1(A))f_2(A))\le w(f(A))$, so it follows that
\begin{align*}
w(f(A))&\ge
\Re \langle \phi_a(f_1(A))f_2(A)y,y\rangle\\
&=\Re \langle (f_1(A)-aI+\overline{a}f_1(A)^2+O(|a|^2))f_2(A)y,y\rangle\\
&=\Re\langle f(A)y,y\rangle-\Re\bigl(a\langle f_2(A)y,y\rangle\bigr)+ \Re\bigl(\overline{a}\langle f_1(A)^2f_2(A)y,y\rangle\bigr)+O(|a|^2)\\
&=w(f(A))-\Re\bigl(\overline{a}\langle y,f_2(A)y\rangle\bigr)+ \Re\bigl(\overline{a}\langle f_1(A)f(A)y,y\rangle\bigr)+O(|a|^2)\\
&=w(f(A))+\Re\Bigl(\overline{a}\bigl(\langle f_1(A)f(A)y,y\rangle-\langle y,f_2(A)y\rangle\bigr)\Bigr)+O(|a|^2).
\end{align*}
Letting $a\to0$, and noting that the argument of $a$ is arbitrary, we obtain \eqref{E:worthog}.
\end{proof}

\subsection{Representation formulas for extremal vectors}\label{sec:rep}

For what follows, recall that $\mathcal{A}(\Omega):=H^\infty(\Omega)\cap C(\overline{\Omega})$.

\begin{theorem}\label{T:rep}
Let $f$ be extremal for $(A,\Omega)$, and let $x$ be an associated extremal unit vector.
Then there exists a unique Borel probability measure $\mu$ on $\partial\Omega$ such that
\begin{equation}\label{E:rep}
\langle h(A)x,x\rangle=\int_{\partial\Omega} h\,d\mu \quad \text{ for all } h\in \mathcal{A}(\Omega).
\end{equation}
If, further, $\|f(A)\| > 1$, then
\[\int_{\partial \Omega} f d\mu = 0.\]
\end{theorem}

\begin{proof}[Proof of Theorem~\ref{T:rep}]
Let $h\in \mathcal{A}(\Omega)$ be a function such that $\Re h\ge0$ on $\overline{\Omega}$.

Fix $t>0$. Since $f$ is  extremal for $(A,\Omega)$, we have
$\|e^{-th(A)}f(A)\|\le \|f(A)\|$. It follows that
\begin{align*}
\|f(A)\|^2&\ge
\Re \langle e^{-th(A)}f(A)x,f(A)x\rangle\\
&=\|f(A)x\|^2-\Re \langle t h(A)f(A)x,f(A)x\rangle+O(t^2)\\
&=\|f(A)\|^2-t\Re \langle  h(A)f(A)x,f(A)x\rangle+O(t^2).
\end{align*}
Cancelling off the two terms $\|f(A)\|^2$, dividing by $t$ and then letting $t\to0^+$, we obtain
\[
0\le \Re\langle h(A)f(A)x,f(A)x\rangle
=\Re\langle h(A)x,f(A)^*f(A)x\rangle
=\|f(A)\|^2\Re \langle h(A)x,x\rangle.
\]
In summary, we have shown that $\Re \langle h(A)x,x\rangle\ge0$ whenever $h\in \mathcal{A}(\Omega)$ and $\Re h\ge0$ on $\overline{\Omega}$.

We now define a linear functional $\Lambda:C(\partial \Omega,\mathbb{R})\to\mathbb{R}$ as follows. If $g$ is of the form $g=\Re h| \partial \Omega$, where $h\in \mathcal{A}(\Omega)$, then we set $\Lambda(g):=\Re\langle h(A)x,x\rangle$.
This is a well-defined, linear functional defined on a dense subspace of $C(\partial\Omega,\mathbb{R})$, and by what we have proved above it maps positive functions to positive numbers. Therefore it extends to a positive linear functional on the whole space, and is given by integration against a finite positive Borel measure $\mu$ on $\partial\Omega$. In particular
\[
\Re \langle h(A)x,x\rangle=\Lambda(\Re h)=\int_{\partial\Omega}(\Re h)\,d\mu=\Re\int_{\partial\Omega}h\,d\mu
 \quad \text{ for all } h\in \mathcal{A}(\Omega).
\]
Repeating with $h$ replaced by $ih$, we have $\Im\langle h(A)x,x\rangle=\Im\int_{\partial\Omega} h\,d\mu$, and so, by linearity, we get \eqref{E:rep}. In particular, for $h=1$,
\[
\mu(\partial\Omega)=\int_{\partial\Omega}1\,d\mu=\langle x,x\rangle=1,
\]
so $\mu$ is a probability measure. Finally, condition \eqref{E:rep} determines $\mu$ uniquely among Borel probability measures on $\partial\Omega$, since the real parts of functions in $\mathcal{A}(\Omega)$ are uniformly dense in $C(\partial\Omega,\mathbb{R})$.

For the final statement, note that though $f$ does not explicitly appear in \eqref{E:rep}, it is linked to $\mu$ by the relation
\[
(\|f(A)\|^2-1)\int_{\partial\Omega}f\,d\mu=0,
\]
as can be seen by combining \eqref{E:orthog2} and \eqref{E:rep}. Thus, if $\|f(A)\| > 1$, we must have \[\int_{\partial \Omega} f d\mu = 0.\]
\end{proof}

The following result gives some more information about the form of $\mu$.
We denote by $\phi$ a bijective conformal mapping of $\Omega$ onto $\mathbb{D}$.
Under our smoothness assumptions on $\Omega$,
the function $\phi$ extends to a diffeomorphism of $\overline{\Omega}$ onto $\overline{\mathbb{D}}$.

\begin{theorem}\label{T:rho}
With the notation of Theorem~\ref{T:rep}, we have
\begin{equation}\label{E:mu}
d\mu(\zeta)=\rho(\phi(\zeta))|\phi'(\zeta)|\,\frac{|d\zeta|}{2\pi},
\end{equation}
where  $\rho$ is
a smooth positive function on $\mathbb{T}$, and is given by
\begin{equation}\label{E:rho}
\rho(e^{i\theta})=
2\Re\Bigl\langle (I-e^{-i\theta}\phi(A))^{-1}x,x\Bigr\rangle-1 \quad(e^{i\theta}\in\mathbb{T}).
\end{equation}
\end{theorem}

\begin{proof}
By the spectral mapping theorem, the spectrum of $\phi(A)$ is contained inside $\mathbb{D}$, and in particular
its spectral radius is strictly less than $1$. Therefore we can expand $(1-e^{-i\theta}\phi(A))^{-1}$ as a power series
to get
\[
\langle (I-e^{-i\theta}\phi(A))^{-1}x,x\rangle=\sum_{k\ge0}e^{-ik\theta}\langle\phi(A)^kx,x\rangle.
\]
Taking real parts of both sides, we obtain
\[
2\Re \langle (I-e^{-i\theta}\phi(A))^{-1}x,x\rangle=\sum_{k\ge0}e^{-ik\theta}\langle\phi(A)^kx,x\rangle+\sum_{k\ge0}e^{ik\theta}\overline{\langle \phi(A)^kx,x\rangle}.
\]
Hence, defining $\rho$ as in \eqref{E:rho}, we have
\begin{equation}\label{E:Kelly}
\int_\mathbb{T} e^{in\theta} \rho(e^{i\theta})\,\frac{d\theta}{2\pi}
=\langle\phi(A)^nx,x\rangle
\quad(n\ge0).
\end{equation}
But also, by \eqref{E:rep}, we have
\[
\langle\phi(A)^nx,x\rangle =\int_{\partial\Omega}\zeta^n\,d\mu(\zeta)=\int_\mathbb{T} e^{in\theta}\,d(\mu\phi^{-1})(e^{i\theta})
\quad(n\ge0).
\]
Thus $\rho\,d\theta/2\pi$ and $\mu \phi^{-1}$ are both real measures on $\mathbb{T}$ that agree on functions of the form $e^{in\theta}$ for $n\ge0$. This implies that  $\rho\,d\theta/2\pi=\mu \phi^{-1}$, giving \eqref{E:mu}.
In particular, as $\mu$ is positive, so is $\rho$.
\end{proof}

\begin{remark}
The formula \eqref{E:Kelly}  is in fact valid for every $x\in\mathbb{C}^n$. However, $\rho$ will not be positive in general.
\end{remark}

\begin{example}\label{ex:comp2}  Let $\Theta$ be a finite Blaschke product with $\deg \Theta =n$.  By Theorems \ref{thm:norm} and \ref{thm:vector}, each Blaschke product with $\deg B <n$ is extremal for $(S_{\Theta}, \mathbb{D})$ with extremal vectors $x$ characterized by $x, Bx \in K_{\Theta}$ with $x$ nonzero. Recall that if $f \in \mathcal{A}(\mathbb{D})$, then $f(S_{\Theta}) = P_{\Theta} T_f |_{K_{\Theta}}$. Thus, we can trivially represent the inner product against the extremal vectors $x \in K_{\Theta}$ via integration, as follows:
\[ \left \langle f (S_{\Theta}) x,x \right \rangle_{K_{\Theta}}  = \left \langle  P_{\Theta} (fx), x\right \rangle_{K_{\Theta}} = \left \langle fx, x \right \rangle_{H^2} =  \frac{1}{2\pi} \int_{\mathbb{T}} f(z) |x(z)|^2 |dz|.\]
Thus, in this case, the measure from Theorem \ref{T:rho} associated to an extremal $x$ is exactly $d\mu=|x(z)|^2 |dz|$.
\end{example}

Just as in the previous section, there is a version of Theorem~\ref{T:rep}  for $w$-extremal functions.

\begin{theorem}\label{T:wrep}
Let $f$ be $w$-extremal for $(A,\Omega)$, and let $y$ be an associated $w$-extremal unit vector.
Then there exists a unique finite positive measure $\nu$ on $\partial\Omega$ such that
\begin{equation}\label{E:wrep}
\langle h(A)f(A)y,y\rangle=\int_{\partial\Omega} h\,d\nu  \quad \text{ for all } h\in \mathcal{A}(\Omega).
\end{equation}
Moreover, we have
\[\nu(\partial \Omega) = w(f(A))~\mbox{and}~\int_{\partial \Omega} f d\nu = 1.\]
\end{theorem}

\begin{proof}
Let $h\in \mathcal{A}(\Omega)$ be a function such that $\Re h\ge0$ on $\overline{\Omega}$.
Fix $t>0$. Since $f$ is  $w$-extremal for $(A,\Omega)$, we have
$w(e^{-th(A)}f(A))\le w(f(A))$. It follows that
\begin{align*}
w(f(A))&\ge
\Re \langle e^{-th(A)}f(A)y,y\rangle\\
&=\Re\langle f(A)y,y\rangle-\Re \langle t h(A)f(A)y,y\rangle+O(t^2)\\
&=w(f(A))-t\Re \langle  h(A)f(A)y,y\rangle+O(t^2).
\end{align*}
Cancelling off the two terms $w(f(A))$, dividing by $t$ and then letting $t\to0^+$, we obtain
\[
\Re\langle h(A)f(A)y,y\rangle\ge0
\]
In summary, we have shown that $\Re \langle h(A)f(A)y,y\rangle\ge0$ whenever $h\in \mathcal{A}(\Omega)$ and $\Re h\ge0$ on $\overline{\Omega}$.
The rest of the proof follows the same route as that of Theorem~\ref{T:rep}.

Notice that, in this case, the measure $\nu$ that we obtain is not a probability measure. Instead we have
\[
\nu(\partial\Omega)=\int_{\partial\Omega}1\,d\nu=\langle f(A)y,y\rangle=w(f(A)).
\]
Moreover, combining Theorem \ref{T:worthog} with \eqref{E:wrep} yields
\[ \int_{\partial\Omega}f\,d\nu=\langle f(A)^2y,y\rangle =1,\]
which completes the proof.
\end{proof}


\section{Appendix} \label{sec:app}
This section contains the proof of Theorem \ref{thm:eigenvectors}. First we need the following lemma:

\begin{lemma} \label{lem:tech} Fix any $m \in \mathbb{N}$ and $a_1, \dots, a_m, b \in \mathbb{C}$. Then
\[ \sum_{k=1}^m\left[ ( 1- |a_k|^2) \prod_{\ell=k+1}^m (1-\bar{a}_\ell b) \prod_{j=1}^{k-1} (|a_j|^2 -\bar{a}_j b ) \right]+ \prod_{j=1}^m (|a_j|^2 - \bar{a}_j b ) = \prod _{j=1}^m (1-\bar{a}_j b).\]
\end{lemma}

\begin{proof} We prove this via induction on $m$. If $m=1$, then the left-hand-side is
\[ (1-|a_1|^2) + (|a_1|^2 -\bar{a}_1 b) = 1-\bar{a}_1b, \]
as needed. Now assuming the result holds for $m$, we will show it holds for $m+1$. Fix $a_1, \dots, a_{m+1}, b \in \mathbb{C}$. Starting with the left-hand-side of the equation and pulling the $m+1$ term out of the sum, we have
\[
\begin{aligned}
&\sum_{k=1}^{m+1} \left[ ( 1- |a_k|^2) \prod_{\ell=k+1}^{m+1} (1-\bar{a}_\ell b) \prod_{j=1}^{k-1} (|a_j|^2 -\bar{a}_j b ) \right]+ \prod_{j=1}^{m+1} (|a_j|^2 - \bar{a}_j b ) \\
&=  \sum_{k=1}^m\left[ ( 1- |a_k|^2) \prod_{\ell=k+1}^{m+1} (1-\bar{a}_\ell b) \prod_{j=1}^{k-1} (|a_j|^2 -\bar{a}_j b ) \right]+ \left[\prod_{j=1}^{m} (|a_j|^2 - \bar{a}_j b ) \right] (1-\bar{a}_{m+1} b) \\
& = (1-\bar{a}_{m+1} b) \left[ \sum_{k=1}^m\left[ ( 1- |a_k|^2) \prod_{\ell=k+1}^m (1-\bar{a}_\ell b) \prod_{j=1}^{k-1} (|a_j|^2 -\bar{a}_j b ) \right]+ \prod_{j=1}^m (|a_j|^2 - \bar{a}_j b ) \right] \\
& = \prod _{j=1}^{m+1} (1-\bar{a}_j b),
\end{aligned}
\]
by the induction hypothesis.\end{proof}

Now we can proceed to the proof of Theorem \ref{thm:eigenvectors}.

\begin{proof} It suffices to show that the columns of $X$ are associated eigenvectors of $M_{\Theta}$, namely $M_{\Theta} \text{Col}_j X = z_j \text{Col}_j{X}$. For simplicity of notation, write $M:= M_{\Theta}$. Then we need to show that
\[  \sum_{k=1}^n M_{ik}X_{kj} = z_j X_{ij}\]
for $1 \le i,j\le n.$ First assume, $i >j$. If $k<i$, then $M_{ik} =0$. If $k \ge i$, then $k >j$ so $X_{kj}=0$. This immediately implies that
\[  \sum_{k=1}^n M_{ik}X_{kj} =  0 = z_j X_{ij},\]
as needed. Similarly, if $i=j$, then $M_{ik}=0$ if $i >k$ and $X_{ki}=0$ if $k>i$. Thus
\[  \sum_{k=1}^n M_{ik}X_{ki}  = M_{ii}X_{ii} = z_i X_{ii}.\]
Now assume $i<j$. If $k<i$, then $M_{ik}=0$. Similarly, if $k >j$, then $X_{kj}=0$. Therefore,
\[  \sum_{k=1}^n M_{ik}X_{kj} = \sum_{k=i}^j M_{ik}X_{kj}  = M_{ii} X_{ij} + \sum_{k=i+1}^{j-1} M_{ik}X_{kj} + M_{ij} X_{jj}.\]
 Using the formulas for $M$ and $X$ gives
 \[
 \begin{aligned}
 \sum_{k=1}^n M_{ik}X_{kj} &= z_i  \frac{ \sqrt{1-|z_i|^2} \sqrt{1-|z_j|^2}}{z_j - z_i}\prod_{\ell=i+1}^{j-1} \left( \frac{1-\bar{z}_\ell z_j}{z_j - z_\ell}\right) \\
&+ \sum_{k=i+1}^{j-1} \sqrt{1-|z_i|^2} \sqrt{1-|z_k|^2} \prod_{\ell=i+1}^{k-1} (-\bar{\lam}_\ell)  \frac{ \sqrt{1-|z_k|^2} \sqrt{1-|z_j|^2}}{z_j - z_k}\prod_{r=k+1}^{j-1} \left( \frac{1-\bar{z}_r z_j}{z_j - z_r}\right)  \\
&+ \sqrt{1-|z_i|^2} \sqrt{1-|z_j|^2} \prod_{\ell=i+1}^{j-1} (-\bar{\lam}_\ell).
\end{aligned}
\]
Factoring out common terms, getting a common denominator, and multiplying the $\bar{z}_{\ell}$ terms through gives
\[
\begin{aligned}
&\frac{ \sqrt{1-|z_i|^2} \sqrt{1-|z_j|^2}}{\prod_{\ell=i}^{j-1} (\lam_j- \lam_\ell)} \Bigg( z_i \prod_{\ell=i+1}^{j-1}(1-\bar{\lam}_{\ell} \lam_j) \\
&+(z_j-z_i) \Bigg(\sum_{k=i+1}^{j-1} \left[(1-|\lam_k|^2) \prod_{\ell=i+1}^{k-1} (|\lam_\ell|^2-\bar{\lam}_\ell \lam_j) \prod_{r=k+1}^{j-1} (1-\bar{\lam}_r \lam_j) \right]+ \prod_{\ell=i+1}^{j-1} (|\lam_\ell|^2-\bar{\lam}_\ell \lam_j) \Bigg) \Bigg).
\end{aligned}
\]
Applying Lemma \ref{lem:tech} with $m = j-1-i$ and $a_1 = \lam_{i+1}, a_2 = \lam_{i+2}, \dots, a_m=\lam_{j-1}$ and $b = \lam_j$ gives
\[\frac{ \sqrt{1-|z_i|^2} \sqrt{1-|z_j|^2}}{\prod_{\ell=i}^{j-1} (\lam_j- \lam_\ell)} \Bigg( z_i \prod_{\ell=i+1}^{j-1}(1-\bar{\lam}_{\ell} \lam_j)  +  (z_j-z_i) \prod_{\ell=i+1}^{j-1}(1-\bar{\lam}_{\ell} \lam_j)\Bigg) = \lam_j X_{ij},\]
as needed.
\end{proof}
A similar argument shows that if we let $T$ denote the second matrix defined in Theorem \ref{thm:eigenvectors}, then $TX = I$, the $n\times n$ identity matrix. This implies that $T=X^{-1}$ and so establishes the formula for $X^{-1}$. We omit the details of this computation here.

\section*{Acknowledgement} The authors of this note are grateful to the American Institute of Mathematics (AIM) and to the organizers of the Workshop on Crouzeix's conjecture for bringing us together. We are also indebted to  Banff International Research Station (BIRS) for providing us with the opportunity to work as a team at BIRS as part of a focused research group. Lastly, we thank \L ukasz Kosi\'nski for helpful discussions.

\end{document}